\documentclass[12pt,a4paper,reqno]{amsart}

\usepackage[english]{babel}
\usepackage[utf8]{inputenc}

\usepackage[headings]{fullpage}
\usepackage{xcolor}
\usepackage[colorlinks]{hyperref}

\DeclareMathOperator{\Ad}{Ad}
\DeclareMathOperator{\ann}{ann}
\DeclareMathOperator{\diag}{diag}
\DeclareMathOperator{\Aut}{Aut}
\DeclareMathOperator{\Der}{Der}
\DeclareMathOperator{\Exp}{Exp}

\DeclareMathOperator{\Ric}{Ric}
\DeclareMathOperator{\scal}{scal}
\DeclareMathOperator{\SO}{SO}
\DeclareMathOperator{\SU}{SU}

\DeclareMathOperator{\tr}{tr}
\DeclareMathOperator{\ZD}{\mathit{ZD}}

\renewcommand{\Im}{\operatorname{Im}}
\newcommand{\mbb}{\mathbb}
\newcommand{\mf}{\mathfrak}
\newcommand{\mrm}{\mathrm}
\renewcommand{\Re}{\operatorname{Re}}
\newcommand{\so}{\mathfrak{so}}

\theoremstyle{plain}
\newtheorem{theorem}{Theorem}[section]
\newtheorem{lemma}[theorem]{Lemma}
\newtheorem{proposition}[theorem]{Proposition}

\theoremstyle{remark}
\newtheorem{remark}[theorem]{Remark}

\numberwithin{equation}{section}

\begin{document}

\title{The geometry of sedenion zero divisors}

\author{Silvio Reggiani}
\address{CONICET and Universidad Nacional de Rosario, ECEN-FCEIA,
  Departamento de Ma\-te\-má\-ti\-ca. Av. Pellegrini 250, 2000
  Rosario, Argentina.}
\email{\href{mailto:reggiani@fceia.unr.edu.ar}{reggiani@fceia.unr.edu.ar}}
\urladdr{\url{http://www.fceia.unr.edu.ar/~reggiani}}

\date{\today}


\keywords{Cayley-Dickson algebras, Sedenion algebra, Zero divisors, Einstein manifolds, Non-negative curvature}

\subjclass[2020]{53C30, 17A20}

\begin{abstract}
  The sedenion algebra $\mathbb S$ is a non-commutative, non-associative, $16$\mbox{-}\nobreak\hspace{0pt}dimensional real algebra with zero divisors. It is obtained from the octonions through the Cayley-Dickson construction. The zero divisors of $\mathbb S$ can be viewed as the submanifold $\mathcal Z(\mathbb S) \subset \mathbb S \times \mathbb S$ of normalized pairs whose product equals zero, or as the submanifold $\operatorname{\mathit {ZD}}(\mathbb S) \subset \mathbb S$ of normalized elements with non-trivial annihilators. We prove that $\mathcal Z(\mathbb S)$ is isometric to the excepcional Lie group $G_2$, equipped with a naturally reductive left-invariant metric. Moreover, $\mathcal Z(\mathbb S)$ is the total space of a Riemannian submersion over the excepcional symmetric space of quaternion subalgebras of the octonion algebra, with fibers that are locally isometric to a product of two round $3$-spheres with different radii. Additionally, we prove that $\operatorname{\mathit {ZD}}(\mathbb S)$ is isometric to the Stiefel manifold $V_2(\mathbb R^7)$, the space of orthonormal $2$-frames in $\mathbb R^7$, endowed with a specific $G_2$-invariant metric. By shrinking this metric along a circle fibration, we construct new examples of an Einstein metric and a family of homogenous metrics on $V_2(\mathbb R^7)$ with non-negative sectional curvature.
\end{abstract} 

\maketitle

\section{Introduction}

The Cayley-Dickson algebras form a sequence of real algebras $\mathbb{A}_n$, defined recursively beginning with $\mathbb{R}$ and doubling in dimension with each iteration. The first members of this family are the familiar real division algebras: $\mbb R = \mbb A_0$, $\mbb C = \mbb A_1$, $\mbb H = \mbb A_2$ and $\mbb O = \mbb A_3$. The next algebra in this sequence is the so-called sedenion algebra $\mbb S = \mbb A_4$, which is often overlooked in comparison to its lower-dimensional relatives due to its lack of certain desirable algebraic properties. Nonetheless, this somewhat enigmatic algebra has long intrigued mathematicians and has recently found applications in fields such as theoretical physics \cite{gillardThreeFermionGenerations2019} and machine learning \cite{saoudMetacognitiveSedenionValuedNeural2020}.

Since $\mbb S$ is not a division algebra, it is interesting to understand the structure of its zero divisors. The topology of the sedenion zero divisors is described by the principal bundle 
\begin{equation*}
  \SU(2) \to G_2 \to V_2(\mbb R^7),
\end{equation*}
where $G_2$ is the excepcional compact Lie group of rank $2$ and $V_2(\mbb R^7)$ is the Stiefel manifold of orthonormal $2$-frames in $\mbb R^7$. Specifically, $G_2$ is homeomorphic to the submanifold $\mathcal Z(\mbb S) \subset \mbb S \times \mbb S$ of normalized sedenion pairs that multiply to zero; $V_2(\mbb R^7)$ is homeomorphic to the submanifold $\ZD(\mbb S) \subset \mbb S$ of sedenions with norm $\sqrt 2$ that have non-trivial annihilators; and, for each $u \in \ZD(\mbb S)$, the fiber $\SU(2)$ corresponds to the sphere of the annihilator subspace of $u$ (see \cite{morenoZeroDivisorsCayleyDickson1998,bissLargeAnnihilatorsCayley2008}). However, little is known regarding the geometry of the sedenion zero divisors.

Both $\mathcal Z(\mbb S)$ and $\ZD(\mbb S)$ carry a natural geometry as submanifolds of $\mbb R^{32}$ and $\mbb R^{16}$, respectively. Furthermore, since the zero divisors of $\mbb S$ are invariant under $\Aut(\mbb S)$, whose connected component is isomorphic to $G_2$, it follows that $\mathcal Z(\mbb S)$ and $\ZD(\mbb S)$ are homogeneous submanifolds. In this article, we study the intrinsic geometry of the zero divisors of $\mbb S$. First, we prove that $\mathcal Z(\mbb S)$ is isometric to $G_2$ with a naturally reductive left-invariant metric, forming the total space of a Riemannian submersion over the exceptional symmetric space $G_2 /{\SO(4)}$, with fibers locally isometric to a product of two round $3$-spheres with different radii.

Next, we analyze the geometry of $\ZD(\mbb S)$, which is isometric to $V_2(\mbb R^7) = G_2 /{\SU(2)}$ with a particular $G_2$-invariant metric. At first glance, the geometry of $\ZD(\mbb S)$ does not seem very interesting; however, by shrinking the metric along a certain circle fibration, we obtain a family $g_r$ ($r > 0$) of $G_2$-invariant metrics on $V_2(\mathbb{R}^7)$, where $g_{\frac{2}{3}}$ represents the original metric. This process reveals several distinguished examples. Specifically, we prove, among other things, that $(V_2(\mbb R^7), g_r)$:
\begin{itemize}
  \item has positive scalar curvature if and only of $r < \frac{20}3$;
  \item is an Einstein manifold if and only if $r = \frac59$;
  \item has non-negative sectional curvature if and only if $0 < r \le \frac49$.
\end{itemize}

These results are quite remarkable, as Einstein metrics and metrics with non-negative curvature are very rare. To the best of our knowledge, the examples presented in this article are new. The known homogeneous Einstein metrics on $V_2(\mathbb R^7) = G_2 / {\SU(2)} = \SO(7) / {\SO(5)}$ are limited to the unique $\SO(7)$-invariant Einstein metric discovered by Sagle and the so-called Jensen metrics (see \cite{sagleHomogeneousEinsteinManifolds1970,jensenEinsteinMetricsPrincipal1973,backEquivariantGeometryKervaire1987,kerrNewExamplesHomogeneous1998}). It is worth noticing that the metric $g_{\frac59}$ is neither $\SO(7)$-invariant nor a Jensen metric. Regarding metrics with non-negative sectional curvature, we refer to the survey \cite{zillerExamplesRiemannianManifolds2007}. Typically, examples of homogeneous metrics with non-negative curvature appear as normal homogeneous metrics or are constructed through a Cheeger deformation of a metric already known to have non-negative curvature. Recall that none of the metrics $g_r$ is normal homogeneous (nor even naturally reductive) and that the initial metric $g_{\frac23}$ does not possess non-negative sectional curvature.

Let us comment briefly on the proof our main results. In order to study the geometry of $\mathcal Z(\mbb S)$, it is necessary to ``fix an origin'' so that the metric can be identified with a left-invariant metric on $G_2$. Any choice of such an origin for $\mathcal Z(\mbb S)$ leads to isometric metrics on $G_2$, but a well-chosen origin can greatly simplify computations. We select the origin from among the so-called 84 standard zero divisors of $\mbb S$. Then, using the results in \cite{datriNaturallyReductiveMetrics1979}, we show that $\mathcal Z(\mbb S)$ is a naturally reductive space. A similar approach applies to the study of $\ZD(\mbb S)$ with the metric $g_r$. Here, we select another standard zero divisor (different from the previous one) so that the isotropy subgroup of $G_2$ acts trivially on the usual subalgebra
$\mbb H \subset \mbb S$. This choice allows the metric $g_r$ to be expressed in diagonal form with respect to the normal homogeneous metric, making it possible to derive a nice expression for the Ricci tensor of $g_r$.

The most challenging part is to determine the sign of the sectional curvatures of $g_r$. Since there is no manageable expression for the curvature (as there is for naturally reductive spaces), using algebraic manipulation proves to be nearly impossible. Indeed, the sectional curvature function $F_r$ of $g_r$ can be interpreted as a homogeneous polynomial of degree $4$ in $22$ real variables, which, for a generic $r$, has 285 non-trivial coefficients. To show that $F_r$ is non-negative for $0 < r \le \frac49$, we reduce the problem to proving that both $F_0$ (the formal extension of $F_r$ at $r = 0$) and $F_{\frac49}$ are non-negative. By using convex optimization techniques, we are able to prove the stronger result that $F_0$ and $F_{\frac49}$ are polynomial sums of squares. 

Finally we want to mention that the computations required in the proof of some of our results are often cumbersome and were computer checked using the software SageMath. The code used to verify our results is available at \cite{reggianiAuxiliaryComputationsArticle2024}. 

We believe this work shows that the study of the geometry of Cayley-Dickson algebras, particularly regarding their zero divisors, deserves further attention, as it may have interesting implications in differential geometry of compact homogeneous spaces.

\subsection*{Acknowledgements}

This work is supported by CONICET and partially supported by SeCyT-UNR and ANPCyT. The author would like to thank Andreas Arvanitoyeorgos for helpful discussions on homogeneous Einstein metrics on Stiefel manifolds.

\section{Preliminaries and notation}

The main references for this section are \cite{morenoZeroDivisorsCayleyDickson1998,bissLargeAnnihilatorsCayley2008} on Cayley-Dickson algebras and their zero divisors, \cite{arvanitogeorgosIntroductionLieGroups2003} on the geometry of homogeneous spaces and \cite{datriNaturallyReductiveMetrics1979} on naturally reductive left-invariant metrics on compact Lie groups. Observe that in this section, as well as throughout the rest of the article, we start counting indices from 0.

\subsection{Cayley-Dickson algebras}

The Cayley-Dickson algebras $\mbb A_n$ are a family of real algebras, equipped  with an involution $a \mapsto a^*$ (also called conjugation), which are recursively defined starting from $\mbb A_0 = \mbb R$, where $a^* = a$. Each subsequent algebra is defined by setting $\mbb A_n = \mbb A_{n-1} \times \mbb A_{n-1}$ as a vector space, with multiplication given by
\begin{equation*}
  (a, b)(c, d) = (ac - d^*b, da + bc^*)  
\end{equation*}
and involution defined by
\begin{equation*}
  (a, b)^* = (a^*, -b).
\end{equation*}

Notice that the inclusion $a \mapsto (a, 0)$ is a monomorphism of algebras from $\mbb A_{n-1}$ into $\mbb A_n$ for all $n \ge 0$. It is well known that the first four algebras in the Cayley-Dickson construction are the real division algebras $\mbb R$, $\mbb C$, $\mbb H$ and $\mbb O$, respectively. It is also known that the Cayley-Dickson algebras lose some important properties with each iteration. For example, $\mbb A_n$ is commutative if and only if $n \le 1$, associative if and only if $n \le 2$; \emph{alterative} (i.e., $x(xy) = (xx)y$ and $(xy)y = x(yy)$ for all $x, y \in \mbb A_n$) if and only of $n \le 3$. On the other hand, every Cayley-Dickson algebra is \emph{flexible} (i.e., $x(yx) = (xy)x$ for all $x, y \in \mbb A_n$) and \emph{power associative} (i.e., $x^k$ is well defined for all $x \in \mbb A_n$ and $k \in \mbb N$).

For $x \in \mbb A_n$ we define its real and imaginary parts as $\Re x = \frac12 (x + x^*)$ and $\Im x = \frac12 (x - x^*)$, respectively. We say that $x$ is \emph{real} (resp.\ \emph{imaginary}) if $\Im x = 0$ (resp.\ $\Re x = 0$). Thus, one can recover the usual inner product on $\mbb A_n \simeq \mbb R^{2^n}$ by 
\begin{equation*}
  \langle x, y \rangle = \Re(x y^*).
\end{equation*}

In $\mbb A_n$, one has that $\|x\|^2 = xx^*$ for all $x$. However, the identity $\| x y \| = \|x\| \|y\|$ does not hold in general if $n \ge 4$. Recall that $\mbb A_n$ is a division algebra if and only $n \le 3$. If $n \ge 4$, then $\mbb A_n$ has zero divisors. Since, $xy = 0$ implies $yx = 0$, the left and right zero divisors of $\mbb A_n$ coincide. Thus, an element $0 \neq u \in \mbb A_n$ is a \emph{zero divisor} if and only if $\ann u \neq 0$, where $\ann u$ is the kernel of the $\mbb R$-linear map $L_u: \mbb A_n \to \mbb A_n$ given by $L_u(x) = ux$. In \cite{morenoZeroDivisorsCayleyDickson1998}, it is proven that a zero divisor $u$ must be imaginary and $\dim(\ann u) \equiv 0 \mod 4$. Furthermore, in \cite{bissLargeAnnihilatorsCayley2008} it is proven that $\dim \ann u \le 2^n -4n + 4$. One can study the zero divisors globally by defining the sets 
\begin{align*}
  \mathcal Z(\mathbb A_n) & = \{(u, v) \in \mbb A_n \times \mbb A_n : \|u\| = \|v\| = \sqrt 2 \text{ and } uv = 0\}, \\
  \ZD(\mbb A_n) & = \{u \in \mbb A_n: (u, v) \in \mathcal Z(\mbb A_n) \text{ for some } v \in \mbb A_n\}.
\end{align*}
Normalizing the zero divisors to $\sqrt 2$ is not particularly important, but it will be convenient later. When $n \ge 5$, the sets $\ZD_k(\mbb A_n) = \{u \in \ZD(\mbb A_n): \dim(\ann u) = k\}$ are also of interest.

For $n \ge 4$, one has that the automorphism group of $\mbb A_n$ is given by
\begin{equation*}
  \Aut(\mbb A_n) \simeq \Aut(\mbb A_{n-1}) \times S_3 \simeq G_2 \times (S_3)^{n-3},
\end{equation*}
where $S_3$ is the symmetric group in three elements, and $G_2 = \Aut(\mbb O)$ is the $14$\mbox{-}\nobreak\hspace{0pt}dimensional compact simple Lie group of rank $2$. Recall that $G_2$ acts diagonally on $\mbb A_n$. It follows that $\Der(\mbb A_n) = \mf g_2$ for all $n \ge 4$, where $\mf g_2$ is the Lie algebra of $G_2$.

\subsection{Sedenion zero divisors}

From now on we denote the \emph{sedenion algebra} $\mbb A_4$ by $\mbb S$. Let us denote by $e_0, \ldots, e_{15}$ the canonical basis of $\mbb S$. By making an abuse of notation, we also denote by $e_0, \ldots, e_3$ and $e_0, \ldots, e_7$ the canonical basis of $\mbb H$ and $\mbb O$ respectively. The zero divisors of $\mbb S$ have the following form.

\begin{proposition}[See \cite{bissLargeAnnihilatorsCayley2008}]
  \label{prop:zero-divisors}
  An element $(a, b) \in \mbb S$ is a zero divisor if and only if $a, b$ are imaginary elements of $\mbb O$ such that $\|a\| = \|b\| \neq 0$ and $a \perp b$.
\end{proposition}

From this result, one can construct the $84$ \emph{standard zero divisors} of $\mbb S$. Namely, the elements of the form $(e_i + e_j, e_k \pm e_l) \in \mathcal Z(\mbb S)$ such that $1 \le i \le 6$, $9 \le j \le 15$, $i < k \le 7$ and $9 \le l \le 15$ (see Table \ref{tab:standard-84-zer-div}). Clearly, every automorphism of $\mbb S$ maps $\mathcal Z(\mbb S)$ into itself. Moreover, we have that the connected component of $\Aut(\mbb S)$ acts simply and transitively on $\mathcal Z(\mbb S)$:

\begin{theorem}[\cite{morenoZeroDivisorsCayleyDickson1998}]
  \label{thm:moreno}
   $\mathcal{Z}(\mbb S)$ is homeomorphic (and moreover, diffeomorphic) to $G_2$. 
\end{theorem}

Given $(u_0, v_0) \in \mathcal Z(\mbb S)$, we have that $G_2 \cdot u_0 = \ZD(\mbb S)$. It is not difficult to see that the isotropy subgroup at $u_0$ is isomorphic to $\SU(2)$. Note that $G_2 / {\SU(2)}$ is diffeomorphic to the Stiefel manifold $V_2(\mbb R^7)$. In fact, every automorphism of $\mbb O$ is completely determined by its values $a, b, c$ at $e_1, e_2, e_4$ respectively. Here $(a, b, c)$ can be any triple of pairwise orthonormal imaginary octonions of norm 1 such that $ab \perp c$. Hence the map $(a, b, c) \mapsto (a, b)$ identifies with a transitive action of $G_2$ in $V_2(\mbb R^7)$, whose isotropy subgroup at $(e_1, e_2)$ are the octonion automorphism that act trivially on $\mathbb H$, and therefore are isomorphic to $\SU(2)$. Thus, the topology of the sedenion zero divisors is encoded by the principal bundle
\begin{equation*}
  \SU(2) \to G_2 \to V_2(\mbb R^7).
\end{equation*}

\subsection{The Lie algebra of \texorpdfstring{$G_2$}{G2}}\label{sec:lie-algebra-g2}

We think of the Lie group $G_2 = \Aut(\mbb O)$ as a subgroup of $\SO(8)$ in the natural way (since every automorphism of $\mbb O$ fixes $e_0$, we have that $G_2$ is actually a subgroup of $\SO(7)$, but we do not use this identification here). So, we have $\mf g_2$ as a subalgebra of $\mf{so}(8)$. Let us consider the bi-invariant metric $g_{\mathrm{bi}}$ induced by the inner product on $\mf g_2$, which we denote with the same symbol, given by
\begin{equation*}
  g_{\mrm{bi}}(X, Y) = -\tr(XY).
\end{equation*}

Let us denote by $E_{ij} \in \mf{so}(8)$, where $0 \le i < j \le 7$, the matrix such that $(E_{ij})_{ij} = -(E_{ij})_{ji} = -1$ and $(E_{ij})_{kl} = 0$ in any other case. We define
\begin{align*}
X_0 & = \tfrac12 (E_{45} + E_{67}), & X_7 & = \tfrac12 (E_{16} + E_{25}), \\
X_1 & = \tfrac12 (E_{46} - E_{57}), & X_8 & = -\tfrac12 (E_{15} - E_{26}), \\
X_2 & = \tfrac12 (E_{47} + E_{56}), & X_9 & = \tfrac12 (E_{14} + E_{27}), \\
X_3 & = -\tfrac{\sqrt 3}6 (2 E_{23} - E_{45} + E_{67}), & X_{10} & = \tfrac{\sqrt 3}6 (E_{16} - E_{25} + 2 E_{34}), \\
X_4 & = \tfrac{\sqrt 3}6 (2 E_{13} + E_{46} + E_{57}), & X_{11} & = \tfrac{\sqrt 3}6 (E_{17} + E_{24} + 2 E_{35}), \\
X_5 & = -\tfrac{\sqrt 3}6 (2 E_{12} - E_{47} + E_{56}), & X_{12} & = -\tfrac{\sqrt 3}6 (E_{14} - E_{27} - 2 E_{36}), \\
X_6 & = -\tfrac12 (E_{17} - E_{24}), & X_{13} & = -\tfrac{\sqrt 3}6 (E_{15} + E_{26} - 2 E_{37}). \\
\end{align*}

One can see that $X_0, \ldots, X_{13}$ is an orthonormal basis of $\mf g_2$ with respect to the bi-invariant metric. We will denote by $X^0, \ldots, X^{13}$ its dual basis. Define
\begin{align*}
  \mf k_0 = \bigoplus_{i=0}^2 \mbb R X_i, &&
   \mf m_0 = \bigoplus_{i=3}^5 \mbb R X_i, &&
   \mf m_1 = \bigoplus_{i=6}^9 \mbb R X_i, &&
   \mf m_2 = \bigoplus_{i=10}^{13} \mbb R X_i.
\end{align*}

We have that $\mf k_0$ and $\mf m_0$ are two subalgebras of $\mf g_2$ isomorphic to $\mf{so}(3)$ such that $[\mf k_0, \mf m_0] = 0$. Moreover, $\mf k_0 \oplus \mf m_0 \simeq \mf{so}(4)$ is the subalgebra of a maximal subgroup of $G_2$ isomorphic to $\SO(4)$ (cfr.\ \cite{burnessLengthDepthCompact2020}). Such subgroup preserves the orthogonal decomposition $\mbb O = \mbb H \oplus \mbb H^\bot$. Furthermore, the subgroup of $G_2$ with Lie algebra $\mf k_0$ is isomorphic to $\SU(2)$ and acts trivially on $\mbb H$. Recall that $G_2 / {\SO(4)}$, with the normal homogeneous metric, is the symmetric space of quaternion subalgebras of $\mbb O$.

\subsection{Homogeneous and naturally reductive spaces}

Let $G$ be a Lie group and $H$ be a compact subgroup of $G$. Let us denote by $\mf g$ and $\mf h$ the Lie algebras of $G$ and $H$, respectively. Assume that $G$ acts almost effectively on $M = G/H$ and that $M$ is endowed with a $G$-invariant metric $g$. Recall that every $X \in \mf g$ induces a Killing vector field $X^*$ on $M$ defined as $X^*_q = \frac d{dt}\big|_0 \Exp(tX) \cdot q$. The map $X \mapsto X^*$ from $\mf g$ into $\mf X(M)$ satisfies 
\begin{equation*}
  [X, Y]^* = -[X^*, Y^*].
\end{equation*}

Let us fix a reductive decomposition $\mf g = \mf h \oplus \mf m$ (i.e., $\mf m$ is an $\Ad(H)$-invariant subspace of $\mf g$ complementary to $\mf h$), which always exists since $H$ is compact. Assume that $H$ is the isotropy subgroup of $p \in M$. Then we can identify $\mf m \simeq T_pM$. The geometry of $M$ is determined by an $\Ad(H)$-invariant inner product on $\mf m$, which we also denote by $g$, defined such that the map $X \in \mathfrak m \mapsto X^*_p \in T_pM$ is a linear isometry. With this setting, we can compute the Levi-Civita connection of $M$ as 
\begin{equation} \label{eq:LC-connection}
(\nabla_{X^*}Y^*)_p = -\frac12 [X, Y]_{\mf m} + U(X, Y), \qquad X, Y \in \mf m,  
\end{equation}
where $U$ is the algebraic tensor on $\mf m$ given by
\begin{equation*}
  2 g(U(X, Y), Z) = g([Z, X]_{\mf m}, Y) + g(X, [Z, Y]_{\mf m}), \qquad X, Y, Z \in \mf m. 
\end{equation*}

Let $R_{X, Y} = \nabla_{[X, Y]} - [\nabla_X, \nabla_Y]$ be the curvature tensor of $M$. The sectional curvature of $M$ is determined by
\begin{align*}
  g(R_{X, Y} X, Y) = & -\frac34 \|[X, Y]_{\mf m}\|^2 - \frac12 g([X, [X, Y]_{\mathfrak m} ]_{\mathfrak m}, Y) -\frac12 g([Y, [Y, X]_{\mathfrak m} ]_{\mathfrak m}, X) \\
  & + \|U(X, Y)\|^2 - g(U(X, X), U(Y, Y)) + g(Y, [ [X, Y]_{\mathfrak h}, X]_{\mathfrak m})
\end{align*}
for $X, Y \in \mf m$. Also, the Ricci tensor of $M$ is determined by   
\begin{align}
  \Ric(X, X) = & -\frac12 \sum_i \{\| [X, X_i]_{\mathfrak m}\|^2 + g([X, [X, X_i]_{\mathfrak m} ]_{\mathfrak m}, X_i) + 2 g([ X, [X, X_i]_{\mathfrak h} ]_{\mathfrak m}, X_i)\} \notag \\
   & + \frac14 \sum_{i, j} g([X_i, X_j]_{\mathfrak m}, X)^2 
  - g([Z, X]_{\mathfrak m}, X), \label{eq:Ricci-tensor}
\end{align}
for $X \in \mf m$, where $\{X_i\}$ is an orthonormal basis of $\mf m$ and $Z = \sum_i U(X_i, X_i)$.

Recall that the metric $g$ on $M = G/H$ is \emph{naturally reductive} if and only if $U \equiv 0$. An interesting particular case is when a left-invariant metric on a Lie group is naturally reductive (with respect to a certain transitive Lie group of isometries).

\begin{theorem}[\cite{datriNaturallyReductiveMetrics1979}]
  \label{thm:datri-ziller}
  Let $G$ be a compact, simple Lie group group endowed with a left-invariant metric $g$. Let $\mf g$ denote the Lie algebra of $G$ and let $g_{\mrm{bi}}$ be a bi-invariant metric on $G$ (which is a negative multiple of the Killing form of $\mf g$). The metric $g$ is naturally reductive if and only if there exists a subalgebra $\mf k$ of $\mf g$ such that
  \begin{equation*}
    g = g_{\mf k_0} \oplus \alpha_1 \, g_{\mrm{bi}}|_{\mf k_1} \oplus \cdots \oplus \alpha_r \, g_{\mrm{bi}}|_{\mf k_r} \oplus \alpha \, g_{\mrm{bi}}|_{\mf k^\bot} 
  \end{equation*}
  where $\mf k = \mf k_0 \oplus \mf k_1 \oplus \cdots \oplus \mf k_r$, with $\mf k_0$ the center of $\mf k$ and $\mf k_1, \ldots, \mf k_r$ are simple ideals. Here, $\mf k^\bot$ is the orthogonal complement of $\mf k$ with respect to the bi-invariant metric, $g_{\mf k_0}$ is an arbitrary inner product on $\mf k_0$, and $\alpha_1, \ldots, \alpha_r, \alpha$ are positive real numbers.
\end{theorem}

\section{The \texorpdfstring{$G_2$}{G2}-invariant metrics on \texorpdfstring{$\mathcal Z(\mbb S)$}{Z(S)} and \texorpdfstring{$\ZD(\mathbb S)$}{ZD(S)}}

Consider $\mathcal{Z}(\mbb S)$ as a submanifold of $\mbb S \times \mbb S \simeq \mbb R^{32}$ with the induced metric. Although this reduction is not necessary here, one could lower the codimension of $\mathcal Z(\mbb S)$. In fact, by Proposition~\ref{prop:zero-divisors}, $\mathcal Z(\mbb{S})$ is a submanifold of $S^6 \times S^6 \times S^6 \times S^6$. Since $G_2 = \Aut(\mbb O) \subset \Aut(\mbb S)$ acts isometrically on $\mbb S \times \mbb S$, we have that $\mathcal Z(\mbb S)$ is a homogeneous submanifold. Furthermore, by Theorem~\ref{thm:moreno}, the diffeomorphism $\mathcal{Z}(\mbb S) \simeq G_2$ induces a left-invariant metric $g$ on $G_2$.

\begin{theorem}
  The metric on $\mathcal Z(\mbb S)$ is naturally reductive. Furthermore, $\mathcal Z(\mbb S)$ is the total space of a Riemannian submersion over the excepcional symmetric space $G_2 /{\SO(4)}$ with totally geodesic fibers, which are locally isometric to a product of two round $3$-spheres with different radii. 
\end{theorem}

\begin{proof}
  It is sufficient to prove the theorem for the left-invariant metric on $G_2$ defined in the paragraph preceding the statement. To determine such a metric, one fixes an element $(u_0, v_0) \in \mathcal Z(\mbb S)$ and computes 
  \begin{equation}\label{eq:1}
    g(X_i, X_j) = (X_i \cdot (u_0, v_0))^T (X_j \cdot (u_0, v_0)).
  \end{equation}

  Note that not every zero divisor pair behaves nicely with respect to the decomposition $\mf g_2 = \mf k_0 \oplus \mf m_0 \oplus \mf m_1 \oplus \mf m_2$ given in Subsection \ref{sec:lie-algebra-g2}. By running \eqref{eq:1} over the standard zero divisors from Table \ref{tab:standard-84-zer-div}, we observe that if $(u_0, v_0) = (e_4 + e_{13}, e_6 + e_{15})$ the metric can be expressed as
  \begin{align*}
    g & = \sum_{i=0}^2 X^i \otimes X^i + \frac13 \sum_{i=3}^5 X^i \otimes X^i + \frac12 \sum_{i=6}^{13} X^i \otimes X^i = g_{\mrm{bi}}|_{\mf k_0} \oplus \frac13\,g_{\mrm{bi}}|_{\mf m_0} \oplus \frac12 \, g_{\mrm{bi}}|_{\mf m_1 \oplus \mf m_2}.
  \end{align*}
  From Theorem \ref{thm:datri-ziller}, it follows that this metric is naturally reductive. More precisely, this metric is naturally reductive with respect to $G_2 \times \SO(4)$, where the second factor acts on the right and the isotropy subgroup is given by $\diag(\SO(4) \times \SO(4))$. Thus, from \cite[Theorem 8]{datriNaturallyReductiveMetrics1979}, the subgroup $\SO(4) \subset G_2$, whose Lie algebra is given by $\mf k_0 \oplus \mf m_0$, is totally geodesic.

  Since $g|_{\mf m_1 \oplus \mf m_2}$ is a multiple of the bi-invariant metric, when restricted to $\mf m_1 \oplus \mf m_2$, and $\mf k_0 \oplus \mf m_0 \simeq \mf{so}(4)$ is orthogonal to $\mf m_1 \oplus \mf m_2$ with respect to both metrics, we conclude that $(G_2, g) \to G_2 / {\SO(4)}$ is a Riemannian submersion. The fiber of this submersion is isometric to the Lie group $\SO(4)$ endowed with the bi-invariant metric $g|_{\mf k_0 \oplus \mf m_0}$, which is obtained by taking two different scalings of the bi-invariant metric on the simple ideals $\so(3) \simeq \mf k_0 \simeq \mf m_0$ of $\so(4)$. Hence, the universal cover of $\SO(4)$ splits into a product of two round spheres with different radii.
\end{proof}

\begin{remark}
  Since the metric in $\mathcal Z(\mbb S)$ is naturally reductive, many geometric properties follow from existing results. For example, the (connected component of the) full isometry group is computed in \cite{datriNaturallyReductiveMetrics1979} (see also \cite{olmosNoteUniquenessCanonical2013}). The so-called index of symmetry of $\mathcal{Z}(\mbb S)$, which in this case is trivial, can be computed from the results in \cite{olmosIndexSymmetryCompact2014}. It can also be seen from \cite{datriNaturallyReductiveMetrics1979} that the metric on $\mathcal Z(\mbb S)$ is not Einstein. We verify this fact again in the next proposition by explicitly computing the Ricci tensor, which also allows us to show that the Ricci curvature is positive.
\end{remark}

\begin{proposition}
  $\mathcal Z(\mbb S)$ has positive Ricci curvature. Moreover,
\begin{equation}
    \Ric = \frac52\sum_{i=0}^2 X^i \otimes X^i + \frac{29}{54} \sum_{i=3}^5 X^i \otimes X^i + \frac56 \sum_{i=6}^{13} X^i \otimes X^i. \label{eq:Ricci-G2}
\end{equation}
\end{proposition}

\begin{proof}
  It follows from a straightforward computation using the following well-known formula. Let $Y_0, \ldots, Y_{13}$ be an $g$-orthonormal basis of $\mf g_2$. Then 
  \begin{equation*}
    \Ric(Y_j, Y_h) = \frac12 \sum_{i,k} \left\{ c_{iki} (c_{kjh} + c_{khj}) + \frac12 c_{ikh} c_{ikj} - c_{ijk} c_{khi} + c_{iki} c_{jhk} - c_{ijk} c_{ihk} \right\}
  \end{equation*}
  where $c_{ijk} = g([Y_i, Y_j], Y_k )$. Since $g$ has diagonal form in the basis $X_0, \ldots, X_{13}$, we can choose $Y_i = g(X_i, X_i)^{-\frac12} X_i$. From this, we can show that
  \begin{equation*}
    \Ric = \frac52\sum_{i=0}^2 Y^i \otimes Y^i + \frac{29}{18} \sum_{i=3}^5 Y^i \otimes Y^i + \frac53 \sum_{i=6}^{13} Y^i \otimes Y^i,
  \end{equation*}
  which is equivalent to \eqref{eq:Ricci-G2}.
\end{proof}

Now we direct our attention to the geometry of $\ZD(\mbb S)$ with the metric induced from the ambient space $\mbb S \simeq \mbb R^{16}$. Since $G_2$ acts isometrically and transitively on $\ZD(\mbb S)$, we have that $\ZD(\mbb S)$ is isometric to the Stiefel manifold $G_2 \cdot u_0 = G_2 /{\SU(2)} = V_2(\mbb R^7)$, equipped with a certain $G_2$-invariant metric, where $\SU(2)$ is the isotropy subgroup of $u_0 \in \ZD(\mbb S)$. We again denote by $g$ such a metric, which is defined by 
\begin{equation*}
  g(X_i, X_j) = (X_i \cdot u_0)^T (X_j \cdot u_0).
\end{equation*}

Similarly to the case of $\mathcal Z(\mbb S)$, we can choose $u_0$ appropriately so that the Lie algebra of $\SU(2)$ is $\mf k_0$. Taking $u_0 = e_1 + e_{10}$, we obtain that
\begin{equation*}
  \mf g_2 = \mf k_0 \oplus \mf m, \qquad \text{where }\mf m = \mf m_0 \oplus \mf m_1 \oplus \mf m_2,
\end{equation*}
is a reductive decomposition for $G_2 / {\SU(2)}$. The corresponding $\Ad(\SU(2))$-invariant inner product on $\mf m$ is given by
\begin{align*}
  g & = \frac13 (X^3 \otimes X^3 + X^4 \otimes X^4) + \frac23 X^5 \otimes X^5 + \frac12 \sum_{i=6}^9 X^i \otimes X^i + \frac16 \sum_{i=10}^{13} X^i \otimes X^i.
\end{align*}

A detailed study of the isometry group and the curvature of $g$ is given in the next section. Before proceeding, we note a simple fact about the sectional curvatures of $g$.

\begin{remark}
  Let us denote by $\pi_{ij}$ the $2$-dimensional subspace of $\mf m$ generated by $X_i$ and $X_j$, where $3 \le i < j \le 13$. Then the sectional curvature of $\pi_{ij}$ is non-negative if and only if $\pi_{ij} \neq \pi_{34}$. This suggests that one could attempt to modify the metric $g$ along the direction normal to $\pi_{34}$ inside $\mf m_1$ in order to get some examples of metrics with non-negative sectional curvature. We explore this approach in the next section. 
\end{remark}

\section{A family of \texorpdfstring{$G_2$}{G2}-invariant metrics on \texorpdfstring{$V_2(\mathbb R^7)$}{V2R7}}

For each $r > 0$, we consider on $V_2(\mbb R^7)$ the family of $G_2$-invariant metrics given by  
\begin{equation}\label{eq:metric-stiefel}
  g_r = \frac13 \, (X^3 \otimes X^3 + X^4 \otimes X^4) + r \, X^5 \otimes X^5 + \frac12 \sum_{i=6}^9 X^i \otimes X^i + \frac16 \sum_{i=10}^{13} X^i \otimes X^i.
\end{equation}

Indeed, $g_r$ gives an $\Ad(\SU(2))$-invariant inner product on $\mf m$ since $\mf m_0$ is the subspace of fixed points of the isotropy representation of $G_2 / {\SU(2)}$ and 
\begin{align*}
  g_r |_{\mf m_1} = \frac12 \, g_{\mrm{bi}} |_{\mf m_1}, && g_r |_{\mf m_2} = \frac16 \, g_{\mrm{bi}} |_{\mf m_2}.
\end{align*}

Next, we compute the connected component of the full isometry group of $g_r$.

\begin{theorem}
  $I_0(V_2(\mbb R^7), g_r) \simeq G_2 \times S^1$.
\end{theorem}

\begin{proof}
  Since $G_2$ is a compact simple Lie group, it follows from the results in \cite{onishchikGroupIsometriesCompact1992} that $I_0(V_2(\mbb R^7), g_r) \subset I_0(V_2(\mbb R^7), g_{\mrm{nh}})$, where $g_{\mrm{nh}} = g_{\mrm{bi}} |_{\mf m}$ is the normal homogeneous metric associated with the homogeneous presentation $V_2(\mbb R^7) = G_2 / {\SU(2)}$. From \cite{reggianiAffineGroupNormal2010}, we have that $I_0(V_2(\mbb R^7), g_{\mrm{nh}}) \simeq G_2 \times K$ (almost direct product) where the Lie algebra of $K$ is given by the $G_2$-invariant vector fields, which are identified with the fixed vectors of the isotropy representation. That is, the Lie algebra of $K$ is identified with $\mf m_0$, but the elements of $K \simeq \SU(2)$ act ``on the right''. Then, it is not difficult to see that $I_0(V_2(\mbb R^7)) \simeq G_2 \times K'$ (almost direct product) for a compact and connected subgroup $K'$ of $K$, which in principle depends on $r$. Since $\dim K = 3$, it is enough to see that $K' \neq K$ and $\dim K' \ge 1$.

  Now, for $Y \in \mf m_0$, let $\hat Y$ be the $G_2$-invariant vector field induced by $Y$. Using \eqref{eq:LC-connection} and the fact that $\nabla_{X^*} \hat Y = \nabla_{\hat Y} X^*$ for all $X \in \mf m$, one can see that $\hat Y$ is a Killing field for $g_r$ if and only if $[Y, -]_{\mf m}: \mf m \to \mf m$ is skew-symmetric with respect to $g_r$. The $\frac12$- and $\frac16$-scalings of the metric on the irreducible subspaces $\mf m_1$ and $\mf m_2$ prevent $\hat X_3$ and $\hat X_4$ from being Killing fields for $g_r$. However, one can check that $\hat X_5$ is a Killing field for $g_r$ for any $r > 0$. Thus $K' = S^1$, which implies $I_0(V_2(\mbb R^7)) = G_2 \times S^1$ is actually a direct product. Observe that we have proved that the $S^1$ factor is independent of $r$.
\end{proof}

Now, we compute the Ricci and scalar curvature of $g_r$. In particular, we obtain the following result.

\begin{theorem}\label{thm:Einstein-scal}
  \begin{enumerate}
    \item The metric $g_r$ is Einstein if and only if $r = \frac59$.
    \item The metric $g_r$ has positive scalar curvature if and only if $r < \frac{20}3$.
  \end{enumerate}
\end{theorem}

\begin{proof}
  Let $Y_3, \ldots, Y_{13}$ be the $g_r$-orthonormal basis of $\mf m$ obtained from normalizing the basis $X_3, \ldots, X_{13}$. We can use formula \eqref{eq:Ricci-tensor} to explicitly compute the Ricci tensor $\Ric_{g_r}$ of $g_r$. After lengthy computations, carefully verified using a computer (see \cite{reggianiAuxiliaryComputationsArticle2024}), we obtain
  \begin{equation}\label{eq:Ricci-gr}
    \Ric_{g_r} = \frac{15 r}{2} \, Y^5 \otimes Y^5 + \left(- \frac{3r}{2} + 5 \right) \sum_{i \neq 5} Y^i \otimes Y^i.
  \end{equation}
  
  Hence, $g_r$ is Einstein if and only if $r = \frac59$. Also, from \eqref{eq:Ricci-gr} we get that the scalar curvature $\scal_{g_r} = 50 - \frac{15}2r$ is positive if and only if $r < \frac{20}{3}$.
\end{proof}

\begin{remark}
  In \cite{jensenEinsteinMetricsPrincipal1973}, the construction of remarkable examples of Einstein metrics on the base space of certain principal bundles can be found. Such metrics are now known as Jensen metrics. In particular, there exist $G_2$-invariant Einstein metrics on $V_2(\mbb R^7)$, arising from the principal bundle $\SU(2) \to G_2 \to G_2 / {\SU(2)}$, which in our notation takes the form $t^2 \, g_{\mrm{bi}} |_{\mf m_0} \oplus g_{\mrm{bi}} |_{\mf m_1 \oplus \mf m_2}$ for certain values of $t > 0$. Notice that the metric $g_{\frac59}$ from Theorem~\ref{thm:Einstein-scal} is not a Jensen metric. Moreover, it is not even bi-invariant when restricted to $\mf m_0$.
\end{remark}

\begin{theorem}\label{thm:non-negative-curvature}
  The metric $g_r$ has non-negative sectional curvature if and only if $r \le \frac49$.
\end{theorem}

In order to prove our theorem, we will need the following result, which is a particular case of Theorem 1 in \cite{powersAlgorithmSumsSquares1998} (see also \cite{choiSumsSquaresReal1995}).

\begin{lemma}\label{lem:powers-wormann}
  Let $F \in \mbb R[x_0, \ldots, x_n]$ be a homogeneous polynomial of degree $4$. Then $F$ is a (polynomial) sum of squares if and only if there exists a symmetric positive semi-definite matrix $H$ such that
  \begin{equation}\label{eq:gram-matrix}
    F = \boldsymbol{x}^T H \boldsymbol{x}
  \end{equation}
  where $\boldsymbol{x} = (x_0^2, x_0 x_1, \ldots, x_{n-1}x_n, x_n^2)^T$ is the vector of monomials of degree $2$.
\end{lemma}

Let us mention that the vector $\boldsymbol{x}$ has $\frac{(n+2)(n+1)}{2}$ coordinates and the subspace of, not necessarily positive semi-definite, matrices $H$ satisfying \eqref{eq:gram-matrix} has dimension $\frac{(n + 2) (n + 1)^{2} n}{12}$. Thus, finding an exact (positive semi-definite) solution $H$ for equation \eqref{eq:gram-matrix} can be quite difficult, even for relatively small values of $n$.

\begin{proof}[{\proofname} of Theorem \ref{thm:non-negative-curvature}]
  It is not hard to see that if $\pi_{34} = \mbb R X_3 \oplus \mbb R X_4$, then the sectional curvature of the plane $\pi_{34}$ is 
  \begin{equation*}
    \kappa_{g_r}(\pi_{34}) = -\frac94 \, r + 1.
  \end{equation*} 

  Thus, $g_r$ does not have non-negative sectional curvature for $r > \frac49$. Let $Y_3, \ldots, Y_{13}$ be the orthonormal basis of $\mf m$ defined in the proof of Theorem \ref{thm:Einstein-scal} and write
  \begin{align*}
    X = \sum_{i = 3}^{13} x_{i - 3} Y_i, && Y = \sum_{i = 3}^{13} x_{i + 8} Y_i.
  \end{align*} 

  For each $r$, consider the polynomial
  \begin{equation*}
    F_r = g_r(R^{g_r}_{X, Y}X, Y) \in \mbb R[x_0, \ldots, x_{21}],
  \end{equation*}
  where $R^{g_r}$ denotes the curvature tensor of $g_r$. Observe that we can formally extend the polynomial $F_r$ to every $r \in \mbb R$ (even when $g_r$ does not make sense for $r \le 0$). Moreover, from the explicit formula for $F_r$, which can be found in the Appendix \ref{app:formula-Fr}, we see that fixing $x_0, \ldots, x_{21}$, the map $r \mapsto F_r(x_0, \ldots, x_{21})$ defines a linear function on $r$. Thus, it is enough to prove that the polynomials $F_0$ and $F_{\frac49}$ are non-negative. We will use Lemma~\ref{lem:powers-wormann} to prove the stronger statement that $F_0$ and $F_{\frac49}$ are polynomial sums of squares. Since $F_r$ is obtained from computing sectional curvatures, every monomial $x_i x_j x_k x_l$ with non-trivial coefficient in $F_r$ satisfies $0 \le i \le j \le 10 < k \le l \le 21$. Hence, we do not lose generality replacing $\boldsymbol{x}^T$ in Lemma \ref{lem:powers-wormann} with 
  \begin{equation*}
    \boldsymbol{x}^T = (x_0 x_{11}, \ldots, x_0 x_{21}, \ldots, x_{10} x_{11}, \ldots, x_{10} x_{21}, x_0^2, \ldots, x_{21}^2).
  \end{equation*}
  This change substantially reduces the size of the system \eqref{eq:gram-matrix} from $253 \times 253$ to $143 \times 143$. Now we are looking for symmetric positive semi-definite matrices $H_\alpha$ such that
  \begin{align*}
    F_\alpha = \boldsymbol{x}^T H_\alpha \boldsymbol{x}, && \alpha \in \{ 0, \tfrac49\}.
  \end{align*} 

  This is a convex optimization problem, which, thanks to the reduction of the dimension mentioned above, can be successfully solved by the Python solver CVXOPT. We implemented the computer code in SageMath through two instances of SemidefiniteProgram(). However, this only yields numerical solutions, and since the condition of being positive semi-definite is a closed one, an exact solution is not guaranteed. Nonetheless, since the polynomial sums of squares are dense in the set of non-negative polynomials, exact solutions are expected to exist.  Moreover, since $F_r$ has relatively few non-trivial coefficients, one can expect to find sparse solutions $H_0$ and $H_{\frac49}$. This is indeed the case, since rounding the numerical solutions lead us to the exact solution described as follows. 
  
  Define the index subsets
  {\allowdisplaybreaks
  \begin{align*}
    I_{0, -2} = {} &  \{(39, 69), (49, 59), (87, 117), (97, 107)\}, \\
    I_{0, -1} = {} &  \{(1, 11), (1, 39), (1, 59), (1, 87), (1, 107), (11, 49), (11, 69), (11, 97), (11, 117), \\
    &  (39, 49), (39, 117), (49, 107), (59, 69), (59, 97), (69, 87), (87, 97), (107, 117)\}, \\
    I_{0, -\frac12} = {} & \{(3, 17), (3, 33), (3, 88), (3, 100), (4, 7), (4, 44), (4, 56), (4, 111), (5, 10), (5, 15), \\
    & (5, 18), (5, 55), (6, 19), (6, 34), (6, 66), (6, 99), (7, 16), (7, 21), (7, 77), (8, 17), \\
    & (8, 33), (8, 88), (8, 100), (9, 19), (9, 34), (9, 66), (9, 99), (10, 45), (10, 78), (10, 110), \\
    & (14, 19), (14, 34), (14, 66), (14, 99), (15, 45), (15, 78), (15, 110), (16, 44), (16, 56), \\
    & (16, 111), (17, 20), (17, 67), (18, 45), (18, 78), (18, 110), (19, 89), (20, 33), (20, 88), \\ 
    & (20, 100), (21, 44), (21, 56), (21, 111), (33, 67), (34, 89), (37, 47), (37, 71), (37, 95), \\
    & (37, 119), (38, 50), (38, 58), (38, 98), (38, 106), (40, 80), (40, 92), (40, 104), \\ 
    &(40, 116), (43, 53), (43, 73), (43, 93), (43, 113), (44, 77), (45, 55), (47, 61), (47, 85), \\
    & (47, 109), (50, 70), (50, 86), (50, 118), (52, 80), (52, 92), (52, 104), (52, 116), \\
    & (53, 63), (53, 83), (53, 103), (55, 78), (55, 110), (56, 77), (58, 70), (58, 86), (58, 118), \\
    & (61, 71), (61, 95), (61, 119), (63, 73), (63, 93), (63, 113), (64, 80), (64, 92), (64, 104), \\
    & (64, 116), (66, 89), (67, 88), (67, 100), (70, 98), (70, 106), (71, 85), (71, 109), \\ 
    & (73, 83), (73, 103), (76, 80), (76, 92), (76, 104), (76, 116), (77, 111), (83, 93), \\
    & (83, 113), (85, 95), (85, 119), (86, 98), (86, 106), (89, 99), (93, 103), (95, 109), \\
    & (98, 118), (103, 113), (106, 118), (109, 119)\}, \\
    I_{0, \frac12} = {} & \{(3, 3), (3, 8), (3, 20), (3, 67), (4, 4), (4, 16), (4, 21), (4, 77), (5, 5), (5, 45), (5, 78), \\
    & (5, 110), (6, 6), (6, 9), (6, 14), (6, 89), (7, 7), (7, 44), (7, 56), (7, 111), (8, 8), (8, 20), \\
    & (8, 67), (9, 9), (9, 14), (9, 89), (10, 10), (10, 15), (10, 18), (10, 55), (14, 14), (14, 89), \\ 
    & (15, 15), (15, 18), (15, 55), (16, 16), (16, 21), (16, 77), (17, 17), (17, 33), (17, 88), \\
    & (17, 100), (18, 18), (18, 55), (19, 19), (19, 34), (19, 66), (19, 99), (20, 20), (20, 67), \\
    & (21, 21), (21, 77), (33, 33), (33, 88), (33, 100), (34, 34), (34, 66), (34, 99), (37, 37), \\
    & (37, 61), (37, 85), (37, 109), (38, 38), (38, 70), (38, 86), (38, 118), (40, 40), (40, 52), \\
    & (40, 64), (40, 76), (43, 43), (43, 63), (43, 83), (43, 103), (44, 44), (44, 56), (44, 111), \\
    & (45, 45), (45, 78), (45, 110), (47, 47), (47, 71), (47, 95), (47, 119), (50, 50), (50, 58), \\
    & (50, 98), (50, 106), (52, 52), (52, 64), (52, 76), (53, 53), (53, 73), (53, 93), (53, 113), \\
    & (55, 55), (56, 56), (56, 111), (58, 58), (58, 98), (58, 106), (61, 61), (61, 85), (61, 109), \\
    & (63, 63), (63, 83), (63, 103), (64, 64), (64, 76), (66, 66), (66, 99), (67, 67), (70, 70), \\
    & (70, 86), (70, 118), (71, 71), (71, 95), (71, 119), (73, 73), (73, 93), (73, 113), (76, 76), \\
    & (77, 77), (78, 78), (78, 110), (80, 80), (80, 92), (80, 104), (80, 116), (83, 83), \\
    & (83, 103), (85, 85), (85, 109), (86, 86), (86, 118), (88, 88), (88, 100), (89, 89), \\
    & (92, 92), (92, 104), (92, 116), (93, 93), (93, 113), (95, 95), (95, 119), (98, 98), \\
    & (98, 106), (99, 99), (100, 100), (103, 103), (104, 104), (104, 116), (106, 106), \\
    & (109, 109), (110, 110), (111, 111), (113, 113), (116, 116), (118, 118), (119, 119)\}, \\
    I_{0, 1} = {} & \{(1, 1), (1, 49), (1, 69), (1, 97), (1, 117), (11, 11), (11, 39), (11, 59), (11, 87), (11, 107), \\
    & (39, 59), (39, 87), (49, 69), (49, 97), (59, 107), (69, 117), (87, 107), (97, 117) \}, \\
    I_{0, 2} = {} & \{(39, 39), (49, 49), (59, 59), (69, 69), (87, 87), (97, 97), (107, 107), (117, 117)\},
  \end{align*} 
  }
  and let $H_0$ be the symmetric matrix defined as
  \begin{equation*}
    (H_0)_{ij} = \begin{cases}
      a, & (i, j) \in I_{0, a}, \\
      0, & (i, j) \notin I_{0, \pm \frac12} \cup I_{0, \pm 1} \cup I_{0, \pm 2}.
    \end{cases}
  \end{equation*}

  Now it is routine to verify that $H_0$ is positive semi-definite and satisfies $F_0 = \boldsymbol{x}^T H_0 \boldsymbol{x}$. Similarly, if 
  {\allowdisplaybreaks
  \begin{align*}
    I_{\frac49, -1} = {} & \{(39, 69), (39, 107), (49, 59), (49, 117), (59, 87), (69, 97), (87, 117), (97, 107)\}, \\
    I_{\frac49, -\frac12} = {} & \{(3, 17), (3, 33), (3, 88), (3, 100), (4, 7), (4, 44), (4, 56), (4, 111), (5, 10), (5, 15), \\
    & (5, 18), (5, 55), (6, 19), (6, 34), (6, 66), (6, 99), (7, 16), (7, 21), (7, 77), (8, 17), \\
    & (8, 33), (8, 88), (8, 100), (9, 19), (9, 34), (9, 66), (9, 99), (10, 45), (10, 78), (10, 110), \\
    & (14, 19), (14, 34), (14, 66), (14, 99), (15, 45), (15, 78), (15, 110), (16, 44), (16, 56), \\
    & (16, 111), (17, 20), (17, 67), (18, 45), (18, 78), (18, 110), (19, 89), (20, 33), (20, 88), \\
    & (20, 100), (21, 44), (21, 56), (21, 111), (33, 67), (34, 89), (37, 47), (37, 71), (37, 95), \\
    & (37, 119), (38, 50), (38, 58), (38, 98), (38, 106), (40, 80), (40, 92), (40, 104), \\
    & (40, 116), (43, 53), (43, 73), (43, 93), (43, 113), (44, 77), (45, 55), (47, 61), (47, 85), \\
    & (47, 109), (50, 70), (50, 86), (50, 118), (52, 80), (52, 92), (52, 104), (52, 116), \\
    & (53, 63), (53, 83), (53, 103), (55, 78), (55, 110), (56, 77), (58, 70), (58, 86), (58, 118), \\
    & (61, 71), (61, 95), (61, 119), (63, 73), (63, 93), (63, 113), (64, 80), (64, 92), (64, 104), \\
    & (64, 116), (66, 89), (67, 88), (67, 100), (70, 98), (70, 106), (71, 85), (71, 109), \\
    & (73, 83), (73, 103), (76, 80), (76, 92), (76, 104), (76, 116), (77, 111), (83, 93), \\
    & (83, 113), (85, 95), (85, 119), (86, 98), (86, 106), (89, 99), (93, 103), (95, 109), \\ 
    & (98, 118), (103, 113), (106, 118), (109, 119)\}, \\
    I_{\frac49, -\frac13} = {} & \{(2, 22), (13, 23), (25, 35), (26, 46), (27, 57), (28, 68), (29, 79), (30, 90), (31, 101), \\
    & (32, 112)\}, \\
    I_{\frac49, \frac13} = {} & \{(2, 2), (13, 13), (22, 22), (23, 23), (25, 25), (26, 26), (27, 27), (28, 28), (29, 29), \\
    & (30, 30), (31, 31), (32, 32), (35, 35), (46, 46), (57, 57), (68, 68), (79, 79), (90, 90), \\
    & (101, 101), (112, 112)\}, \\
    I_{\frac49, \frac12} = {} & \{(3, 3), (3, 8), (3, 20), (3, 67), (4, 4), (4, 16), (4, 21), (4, 77), (5, 5), (5, 45), (5, 78), \\
    & (5, 110), (6, 6), (6, 9), (6, 14), (6, 89), (7, 7), (7, 44), (7, 56), (7, 111), (8, 8), (8, 20), \\
    & (8, 67), (9, 9), (9, 14), (9, 89), (10, 10), (10, 15), (10, 18), (10, 55), (14, 14), (14, 89), \\
    & (15, 15), (15, 18), (15, 55), (16, 16), (16, 21), (16, 77), (17, 17), (17, 33), (17, 88), \\
    & (17, 100), (18, 18), (18, 55), (19, 19), (19, 34), (19, 66), (19, 99), (20, 20), (20, 67), \\
    & (21, 21), (21, 77), (33, 33), (33, 88), (33, 100), (34, 34), (34, 66), (34, 99), (37, 37), \\
    & (37, 61), (37, 85), (37, 109), (38, 38), (38, 70), (38, 86), (38, 118), (40, 40), (40, 52), \\
    & (40, 64), (40, 76), (43, 43), (43, 63), (43, 83), (43, 103), (44, 44), (44, 56), (44, 111), \\
    & (45, 45), (45, 78), (45, 110), (47, 47), (47, 71), (47, 95), (47, 119), (50, 50), (50, 58), \\
    & (50, 98), (50, 106), (52, 52), (52, 64), (52, 76), (53, 53), (53, 73), (53, 93), (53, 113), \\
    & (55, 55), (56, 56), (56, 111), (58, 58), (58, 98), (58, 106), (61, 61), (61, 85), (61, 109), \\
    & (63, 63), (63, 83), (63, 103), (64, 64), (64, 76), (66, 66), (66, 99), (67, 67), (70, 70), \\
    & (70, 86), (70, 118), (71, 71), (71, 95), (71, 119), (73, 73), (73, 93), (73, 113), (76, 76), \\
    & (77, 77), (78, 78), (78, 110), (80, 80), (80, 92), (80, 104), (80, 116), (83, 83), \\
    & (83, 103), (85, 85), (85, 109), (86, 86), (86, 118), (88, 88), (88, 100), (89, 89), \\
    & (92, 92), (92, 104), (92, 116), (93, 93), (93, 113), (95, 95), (95, 119), (98, 98), \\
    & (98, 106), (99, 99), (100, 100), (103, 103), (104, 104), (104, 116), (106, 106), \\
    & (109, 109), (110, 110), (111, 111), (113, 113), (116, 116), (118, 118), (119, 119)\}, \\
    I_{\frac49, 1} = {} & \{(39, 39), (39, 97), (49, 49), (49, 87), (59, 59), (59, 117), (69, 69), (69, 107), (87, 87), \\
    & (97, 97), (107, 107), (117, 117)\}, \\
  \end{align*}
  }
  then \begin{equation*}
    (H_{\frac49})_{ij} = \begin{cases}
      a, & (i, j) \in I_{\frac49, a}, \\
      0, & (i, j) \notin I_{\frac49, \pm \frac13} \cup I_{\frac49, \pm \frac12} \cup I_{\frac49, \pm 1},
    \end{cases}
  \end{equation*}
  defines a symmetric positive semi-definite matrix satisfying $F_{\frac49} = \boldsymbol{x}^T H_{\frac49} \boldsymbol{x}$. This concludes the proof of the theorem.
\end{proof}

\begin{remark}
  We are not certain if $F_r$ is a polynomial sum of squares for $0 < r < \frac49$. Although it is not needed in the proof of Theorem \ref{thm:non-negative-curvature}, it would be interesting to know if this is indeed the case. 
\end{remark}

\begin{remark}
  The definition \eqref{eq:metric-stiefel} of the metric $g_r$ resembles the construction of Berger spheres from the Hopf fibration by shrinking the metrics along the fibers. Moreover, if we restrict $g_r$ to $\mf m_0 \simeq \mf{su}(2)$, then $S^3_r = (\SU(2), g_r |_{\mf m_0})$ is a Berger sphere. Recall that $S^3_r$ has (strictly) positive sectional curvature if and only if $0 < r < \frac49$.
\end{remark}

\appendix

\section{}

\subsection{Standard zero divisors} 

In this appendix, we include Table \ref{tab:standard-84-zer-div} with the standard zero divisors of the sedenion algebra.

\begin{table}[ht]
  \caption{The $84$ standard zero divisors of $\mbb S$}
  \label{tab:standard-84-zer-div}
  \centering
  $\begin{array}{llll}
  \hline \\[-12pt]
  (e_1+e_{10},e_{4}-e_{15})&(e_1+e_{10},e_{5}+e_{14})&(e_1+e_{10},e_{6}-e_{13})&(e_1+e_{10},e_{7}+e_{12})\\
  (e_1+e_{11},e_{4}+e_{14})&(e_1+e_{11},e_{5}+e_{15})&(e_1+e_{11},e_{6}-e_{12})&(e_1+e_{11},e_{7}-e_{13})\\
  (e_1+e_{12},e_{2}+e_{15})&(e_1+e_{12},e_{3}-e_{14})&(e_1+e_{12},e_{6}+e_{11})&(e_1+e_{12},e_{7}-e_{10})\\
  (e_1+e_{13},e_{2}-e_{14})&(e_1+e_{13},e_{3}-e_{15})&(e_1+e_{13},e_{6}+e_{10})&(e_1+e_{13},e_{7}+e_{11})\\
  (e_1+e_{14},e_{2}+e_{13})&(e_1+e_{14},e_{3}+e_{12})&(e_1+e_{14},e_{4}-e_{11})&(e_1+e_{14},e_{5}-e_{10})\\
  (e_1+e_{15},e_{2}-e_{12})&(e_1+e_{15},e_{3}+e_{13})&(e_1+e_{15},e_{4}+e_{10})&(e_1+e_{15},e_{5}-e_{11})\\
  (e_2+e_{9},e_{4}+e_{15})&(e_2+e_{9},e_{5}-e_{14})&(e_2+e_{9},e_{6}+e_{13})&(e_2+e_{9},e_{7}-e_{12})\\
  (e_2+e_{11},e_{4}-e_{13})&(e_2+e_{11},e_{5}+e_{12})&(e_2+e_{11},e_{6}+e_{15})&(e_2+e_{11},e_{7}-e_{14})\\
  (e_2+e_{12},e_{3}+e_{13})&(e_2+e_{12},e_{5}-e_{11})&(e_2+e_{12},e_{7}+e_{9})&(e_2+e_{13},e_{3}-e_{12})\\
  (e_2+e_{13},e_{4}+e_{11})&(e_2+e_{13},e_{6}-e_{9})&(e_2+e_{14},e_{3}-e_{15})&(e_2+e_{14},e_{5}+e_{9})\\
  (e_2+e_{14},e_{7}+e_{11})&(e_2+e_{15},e_{3}+e_{14})&(e_2+e_{15},e_{4}-e_{9})&(e_2+e_{15},e_{6}-e_{11})\\
  (e_3+e_{9},e_{4}-e_{14})&(e_3+e_{9},e_{5}-e_{15})&(e_3+e_{9},e_{6}+e_{12})&(e_3+e_{9},e_{7}+e_{13})\\
  (e_3+e_{10},e_{4}+e_{13})&(e_3+e_{10},e_{5}-e_{12})&(e_3+e_{10},e_{6}-e_{15})&(e_3+e_{10},e_{7}+e_{14})\\
  (e_3+e_{12},e_{5}+e_{10})&(e_3+e_{12},e_{6}-e_{9})&(e_3+e_{13},e_{4}-e_{10})&(e_3+e_{13},e_{7}-e_{9})\\
  (e_3+e_{14},e_{4}+e_{9})&(e_3+e_{14},e_{7}-e_{10})&(e_3+e_{15},e_{5}+e_{9})&(e_3+e_{15},e_{6}+e_{10})\\
  (e_4+e_{9},e_{6}-e_{11})&(e_4+e_{9},e_{7}+e_{10})&(e_4+e_{10},e_{5}+e_{11})&(e_4+e_{10},e_{7}-e_{9})\\
  (e_4+e_{11},e_{5}-e_{10})&(e_4+e_{11},e_{6}+e_{9})&(e_4+e_{13},e_{6}+e_{15})&(e_4+e_{13},e_{7}-e_{14})\\
  (e_4+e_{14},e_{5}-e_{15})&(e_4+e_{14},e_{7}+e_{13})&(e_4+e_{15},e_{5}+e_{14})&(e_4+e_{15},e_{6}-e_{13})\\
  (e_5+e_{9},e_{6}-e_{10})&(e_5+e_{9},e_{7}-e_{11})&(e_5+e_{10},e_{6}+e_{9})&(e_5+e_{11},e_{7}+e_{9})\\
  (e_5+e_{12},e_{6}-e_{15})&(e_5+e_{12},e_{7}+e_{14})&(e_5+e_{14},e_{7}-e_{12})&(e_5+e_{15},e_{6}+e_{12})\\
  (e_6+e_{10},e_{7}-e_{11})&(e_6+e_{11},e_{7}+e_{10})&(e_6+e_{12},e_{7}-e_{13})&(e_6+e_{13},e_{7}+e_{12})\\[2pt]
  \hline
  \end{array}$
\end{table}

\subsection{Expression for the sectional curvature} \label{app:formula-Fr} 

We write down the polynomial $F_r$ used in the proof of Theorem \ref{thm:non-negative-curvature}.

{\allowdisplaybreaks
\begin{align*}
  F_r = {} & (-\tfrac{9}{4} r + 1) x_{0}^{2} x_{12}^{2} + \tfrac{3}{4} r x_{0}^{2} x_{13}^{2} + \tfrac{1}{2} x_{0}^{2} x_{14}^{2} + x_{0}^{2} x_{14} x_{19} + \tfrac{1}{2} x_{0}^{2} x_{15}^{2} - x_{0}^{2} x_{15} x_{18} + \tfrac{1}{2} x_{0}^{2} x_{16}^{2} \\
  & - x_{0}^{2} x_{16} x_{21} + \tfrac{1}{2} x_{0}^{2} x_{17}^{2} + x_{0}^{2} x_{17} x_{20} + \tfrac{1}{2} x_{0}^{2} x_{18}^{2} + \tfrac{1}{2} x_{0}^{2} x_{19}^{2} + \tfrac{1}{2} x_{0}^{2} x_{20}^{2} + \tfrac{1}{2} x_{0}^{2} x_{21}^{2} \\
  & + (\tfrac{9}{2} r - 2) x_{0} x_{1} x_{11} x_{12} + 2 x_{0} x_{1} x_{14} x_{20} + 2 x_{0} x_{1} x_{15} x_{21} - 2 x_{0} x_{1} x_{16} x_{18} \\ 
  & - 2 x_{0} x_{1} x_{17} x_{19} -\tfrac{3}{2} r x_{0} x_{2} x_{11} x_{13} - x_{0} x_{3} x_{11} x_{14} - x_{0} x_{3} x_{11} x_{19} \\ 
  & + (\tfrac{9}{2} r - 3) x_{0} x_{3} x_{12} x_{17} - x_{0} x_{3} x_{12} x_{20} - x_{0} x_{4} x_{11} x_{15} + x_{0} x_{4} x_{11} x_{18} \\
  & + (-\tfrac{9}{2} r + 3) x_{0} x_{4} x_{12} x_{16} - x_{0} x_{4} x_{12} x_{21} - x_{0} x_{5} x_{11} x_{16} + x_{0} x_{5} x_{11} x_{21} \\
  & + (\tfrac{9}{2} r - 3) x_{0} x_{5} x_{12} x_{15} + x_{0} x_{5} x_{12} x_{18} - x_{0} x_{6} x_{11} x_{17} - x_{0} x_{6} x_{11} x_{20} \\
  & + (-\tfrac{9}{2} r + 3) x_{0} x_{6} x_{12} x_{14} + x_{0} x_{6} x_{12} x_{19} + x_{0} x_{7} x_{11} x_{15} - x_{0} x_{7} x_{11} x_{18} + x_{0} x_{7} x_{12} x_{16} \\
  & + (\tfrac{9}{2} r - 3) x_{0} x_{7} x_{12} x_{21} - x_{0} x_{8} x_{11} x_{14} - x_{0} x_{8} x_{11} x_{19} + x_{0} x_{8} x_{12} x_{17} \\
  & + (-\tfrac{9}{2} r + 3) x_{0} x_{8} x_{12} x_{20} - x_{0} x_{9} x_{11} x_{17} - x_{0} x_{9} x_{11} x_{20} - x_{0} x_{9} x_{12} x_{14} \\
  & + (\tfrac{9}{2} r - 3) x_{0} x_{9} x_{12} x_{19} + x_{0} x_{10} x_{11} x_{16} - x_{0} x_{10} x_{11} x_{21} - x_{0} x_{10} x_{12} x_{15} \\ 
  & + (-\tfrac{9}{2} r + 3) x_{0} x_{10} x_{12} x_{18} + (-\tfrac{9}{4} r + 1) x_{1}^{2} x_{11}^{2} + \tfrac{3}{4} r x_{1}^{2} x_{13}^{2} + \tfrac{1}{2} x_{1}^{2} x_{14}^{2} - x_{1}^{2} x_{14} x_{19} \\ 
  & + \tfrac{1}{2} x_{1}^{2} x_{15}^{2} + x_{1}^{2} x_{15} x_{18} + \tfrac{1}{2} x_{1}^{2} x_{16}^{2} + x_{1}^{2} x_{16} x_{21} + \tfrac{1}{2} x_{1}^{2} x_{17}^{2} - x_{1}^{2} x_{17} x_{20} + \tfrac{1}{2} x_{1}^{2} x_{18}^{2} \\ 
  & + \tfrac{1}{2} x_{1}^{2} x_{19}^{2} + \tfrac{1}{2} x_{1}^{2} x_{20}^{2} + \tfrac{1}{2} x_{1}^{2} x_{21}^{2} -\tfrac{3}{2} r x_{1} x_{2} x_{12} x_{13} + (-\tfrac{9}{2} r + 3) x_{1} x_{3} x_{11} x_{17} \\ 
  & - x_{1} x_{3} x_{11} x_{20} - x_{1} x_{3} x_{12} x_{14} + x_{1} x_{3} x_{12} x_{19} + (\tfrac{9}{2} r - 3) x_{1} x_{4} x_{11} x_{16} - x_{1} x_{4} x_{11} x_{21} \\ 
  & - x_{1} x_{4} x_{12} x_{15} - x_{1} x_{4} x_{12} x_{18} + (-\tfrac{9}{2} r + 3) x_{1} x_{5} x_{11} x_{15} + x_{1} x_{5} x_{11} x_{18} \\ 
  & - x_{1} x_{5} x_{12} x_{16} - x_{1} x_{5} x_{12} x_{21} + (\tfrac{9}{2} r - 3) x_{1} x_{6} x_{11} x_{14} + x_{1} x_{6} x_{11} x_{19} \\
  & - x_{1} x_{6} x_{12} x_{17} + x_{1} x_{6} x_{12} x_{20} + x_{1} x_{7} x_{11} x_{16} + (-\tfrac{9}{2} r + 3) x_{1} x_{7} x_{11} x_{21} \\
  & - x_{1} x_{7} x_{12} x_{15} - x_{1} x_{7} x_{12} x_{18} + x_{1} x_{8} x_{11} x_{17} + (\tfrac{9}{2} r - 3) x_{1} x_{8} x_{11} x_{20} \\
  & + x_{1} x_{8} x_{12} x_{14} - x_{1} x_{8} x_{12} x_{19} - x_{1} x_{9} x_{11} x_{14} + (-\tfrac{9}{2} r + 3) x_{1} x_{9} x_{11} x_{19} + x_{1} x_{9} x_{12} x_{17} \\
  & - x_{1} x_{9} x_{12} x_{20} - x_{1} x_{10} x_{11} x_{15} + (\tfrac{9}{2} r - 3) x_{1} x_{10} x_{11} x_{18} - x_{1} x_{10} x_{12} x_{16} - x_{1} x_{10} x_{12} x_{21} \\
  & + \tfrac{3}{4} r x_{2}^{2} x_{11}^{2} + \tfrac{3}{4} r x_{2}^{2} x_{12}^{2} + \tfrac{3}{4} r x_{2}^{2} x_{14}^{2} + \tfrac{3}{4} r x_{2}^{2} x_{15}^{2} + \tfrac{3}{4} r x_{2}^{2} x_{16}^{2} + \tfrac{3}{4} r x_{2}^{2} x_{17}^{2} + \tfrac{3}{4} r x_{2}^{2} x_{18}^{2} \\
  & + \tfrac{3}{4} r x_{2}^{2} x_{19}^{2} + \tfrac{3}{4} r x_{2}^{2} x_{20}^{2} + \tfrac{3}{4} r x_{2}^{2} x_{21}^{2} -\tfrac{3}{2} r x_{2} x_{3} x_{13} x_{14} -\tfrac{3}{2} r x_{2} x_{4} x_{13} x_{15} -\tfrac{3}{2} r x_{2} x_{5} x_{13} x_{16} \\
  & -\tfrac{3}{2} r x_{2} x_{6} x_{13} x_{17} -\tfrac{3}{2} r x_{2} x_{7} x_{13} x_{18} -\tfrac{3}{2} r x_{2} x_{8} x_{13} x_{19} -\tfrac{3}{2} r x_{2} x_{9} x_{13} x_{20} -\tfrac{3}{2} r x_{2} x_{10} x_{13} x_{21} \\
  & + \tfrac{1}{2} x_{3}^{2} x_{11}^{2} + \tfrac{1}{2} x_{3}^{2} x_{12}^{2} + \tfrac{3}{4} r x_{3}^{2} x_{13}^{2} + \tfrac{1}{2} x_{3}^{2} x_{15}^{2} + \tfrac{1}{2} x_{3}^{2} x_{16}^{2} + (-\tfrac{9}{4} r + 2) x_{3}^{2} x_{17}^{2} + \tfrac{1}{2} x_{3}^{2} x_{18}^{2} \\ 
  & + \tfrac{1}{2} x_{3}^{2} x_{21}^{2} - x_{3} x_{4} x_{14} x_{15} + (\tfrac{9}{2} r - 3) x_{3} x_{4} x_{16} x_{17} + x_{3} x_{4} x_{18} x_{19} - x_{3} x_{4} x_{20} x_{21} \\
  & - x_{3} x_{5} x_{14} x_{16} + (-\tfrac{9}{2} r + 3) x_{3} x_{5} x_{15} x_{17} + x_{3} x_{5} x_{18} x_{20} + x_{3} x_{5} x_{19} x_{21} \\
  & + (\tfrac{9}{2} r - 4) x_{3} x_{6} x_{14} x_{17} - x_{3} x_{7} x_{14} x_{18} + x_{3} x_{7} x_{15} x_{19} + x_{3} x_{7} x_{16} x_{20} \\
  &  + (-\tfrac{9}{2} r + 3) x_{3} x_{7} x_{17} x_{21} + x_{3} x_{8} x_{11}^{2} - x_{3} x_{8} x_{12}^{2} - 2 x_{3} x_{8} x_{15} x_{18} - 2 x_{3} x_{8} x_{16} x_{21} \\
  & + \tfrac{9}{2} r x_{3} x_{8} x_{17} x_{20} + 2 x_{3} x_{9} x_{11} x_{12} + 2 x_{3} x_{9} x_{15} x_{21} - 2 x_{3} x_{9} x_{16} x_{18} -\tfrac{9}{2} r x_{3} x_{9} x_{17} x_{19} \\
  & - x_{3} x_{10} x_{14} x_{21} - x_{3} x_{10} x_{15} x_{20} + x_{3} x_{10} x_{16} x_{19} + (\tfrac{9}{2} r - 3) x_{3} x_{10} x_{17} x_{18} + \tfrac{1}{2} x_{4}^{2} x_{11}^{2} \\
  & + \tfrac{1}{2} x_{4}^{2} x_{12}^{2} + \tfrac{3}{4} r x_{4}^{2} x_{13}^{2} + \tfrac{1}{2} x_{4}^{2} x_{14}^{2} + (-\tfrac{9}{4} r + 2) x_{4}^{2} x_{16}^{2} + \tfrac{1}{2} x_{4}^{2} x_{17}^{2} + \tfrac{1}{2} x_{4}^{2} x_{19}^{2} + \tfrac{1}{2} x_{4}^{2} x_{20}^{2} \\
  & + (\tfrac{9}{2} r - 4) x_{4} x_{5} x_{15} x_{16} + (-\tfrac{9}{2} r + 3) x_{4} x_{6} x_{14} x_{16} - x_{4} x_{6} x_{15} x_{17} + x_{4} x_{6} x_{18} x_{20} \\
  & + x_{4} x_{6} x_{19} x_{21} - x_{4} x_{7} x_{11}^{2} + x_{4} x_{7} x_{12}^{2} - 2 x_{4} x_{7} x_{14} x_{19} + \tfrac{9}{2} r x_{4} x_{7} x_{16} x_{21} \\
  & - 2 x_{4} x_{7} x_{17} x_{20} + x_{4} x_{8} x_{14} x_{18} - x_{4} x_{8} x_{15} x_{19} + (-\tfrac{9}{2} r + 3) x_{4} x_{8} x_{16} x_{20} + x_{4} x_{8} x_{17} x_{21} \\
  & - x_{4} x_{9} x_{14} x_{21} - x_{4} x_{9} x_{15} x_{20} + (\tfrac{9}{2} r - 3) x_{4} x_{9} x_{16} x_{19} + x_{4} x_{9} x_{17} x_{18} + 2 x_{4} x_{10} x_{11} x_{12} \\ 
  & + 2 x_{4} x_{10} x_{14} x_{20} -\tfrac{9}{2} r x_{4} x_{10} x_{16} x_{18} - 2 x_{4} x_{10} x_{17} x_{19} + \tfrac{1}{2} x_{5}^{2} x_{11}^{2} + \tfrac{1}{2} x_{5}^{2} x_{12}^{2} + \tfrac{3}{4} r x_{5}^{2} x_{13}^{2} \\
  & + \tfrac{1}{2} x_{5}^{2} x_{14}^{2} + (-\tfrac{9}{4} r + 2) x_{5}^{2} x_{15}^{2} + \tfrac{1}{2} x_{5}^{2} x_{17}^{2} + \tfrac{1}{2} x_{5}^{2} x_{19}^{2} + \tfrac{1}{2} x_{5}^{2} x_{20}^{2} + (\tfrac{9}{2} r - 3) x_{5} x_{6} x_{14} x_{15} \\ 
  & - x_{5} x_{6} x_{16} x_{17} - x_{5} x_{6} x_{18} x_{19} + x_{5} x_{6} x_{20} x_{21} - 2 x_{5} x_{7} x_{11} x_{12} - 2 x_{5} x_{7} x_{14} x_{20} \\
  & -\tfrac{9}{2} r x_{5} x_{7} x_{15} x_{21} + 2 x_{5} x_{7} x_{17} x_{19} + x_{5} x_{8} x_{14} x_{21} + (\tfrac{9}{2} r - 3) x_{5} x_{8} x_{15} x_{20} \\ 
  & - x_{5} x_{8} x_{16} x_{19} - x_{5} x_{8} x_{17} x_{18} + x_{5} x_{9} x_{14} x_{18} + (-\tfrac{9}{2} r + 3) x_{5} x_{9} x_{15} x_{19} - x_{5} x_{9} x_{16} x_{20} \\ 
  & + x_{5} x_{9} x_{17} x_{21} - x_{5} x_{10} x_{11}^{2} + x_{5} x_{10} x_{12}^{2} - 2 x_{5} x_{10} x_{14} x_{19} + \tfrac{9}{2} r x_{5} x_{10} x_{15} x_{18} \\
  & - 2 x_{5} x_{10} x_{17} x_{20} + \tfrac{1}{2} x_{6}^{2} x_{11}^{2} + \tfrac{1}{2} x_{6}^{2} x_{12}^{2} + \tfrac{3}{4} r x_{6}^{2} x_{13}^{2} + (-\tfrac{9}{4} r + 2) x_{6}^{2} x_{14}^{2} + \tfrac{1}{2} x_{6}^{2} x_{15}^{2} \\
  & + \tfrac{1}{2} x_{6}^{2} x_{16}^{2} + \tfrac{1}{2} x_{6}^{2} x_{18}^{2} + \tfrac{1}{2} x_{6}^{2} x_{21}^{2} + (\tfrac{9}{2} r - 3) x_{6} x_{7} x_{14} x_{21} + x_{6} x_{7} x_{15} x_{20} - x_{6} x_{7} x_{16} x_{19} \\
  & - x_{6} x_{7} x_{17} x_{18} - 2 x_{6} x_{8} x_{11} x_{12} -\tfrac{9}{2} r x_{6} x_{8} x_{14} x_{20} - 2 x_{6} x_{8} x_{15} x_{21} + 2 x_{6} x_{8} x_{16} x_{18} \\
  & + x_{6} x_{9} x_{11}^{2} - x_{6} x_{9} x_{12}^{2} + \tfrac{9}{2} r x_{6} x_{9} x_{14} x_{19} - 2 x_{6} x_{9} x_{15} x_{18} - 2 x_{6} x_{9} x_{16} x_{21} \\ 
  & + (-\tfrac{9}{2} r + 3) x_{6} x_{10} x_{14} x_{18} + x_{6} x_{10} x_{15} x_{19} + x_{6} x_{10} x_{16} x_{20} - x_{6} x_{10} x_{17} x_{21} + \tfrac{1}{2} x_{7}^{2} x_{11}^{2} \\
  & + \tfrac{1}{2} x_{7}^{2} x_{12}^{2} + \tfrac{3}{4} r x_{7}^{2} x_{13}^{2} + \tfrac{1}{2} x_{7}^{2} x_{14}^{2} + \tfrac{1}{2} x_{7}^{2} x_{17}^{2} + \tfrac{1}{2} x_{7}^{2} x_{19}^{2} + \tfrac{1}{2} x_{7}^{2} x_{20}^{2} + (-\tfrac{9}{4} r + 2) x_{7}^{2} x_{21}^{2} \\
  & + x_{7} x_{8} x_{14} x_{15} - x_{7} x_{8} x_{16} x_{17} - x_{7} x_{8} x_{18} x_{19} + (\tfrac{9}{2} r - 3) x_{7} x_{8} x_{20} x_{21} + x_{7} x_{9} x_{14} x_{16} \\
  & + x_{7} x_{9} x_{15} x_{17} - x_{7} x_{9} x_{18} x_{20} + (-\tfrac{9}{2} r + 3) x_{7} x_{9} x_{19} x_{21} + (\tfrac{9}{2} r - 4) x_{7} x_{10} x_{18} x_{21} \\
  & + \tfrac{1}{2} x_{8}^{2} x_{11}^{2} + \tfrac{1}{2} x_{8}^{2} x_{12}^{2} + \tfrac{3}{4} r x_{8}^{2} x_{13}^{2} + \tfrac{1}{2} x_{8}^{2} x_{15}^{2} + \tfrac{1}{2} x_{8}^{2} x_{16}^{2} + \tfrac{1}{2} x_{8}^{2} x_{18}^{2} + (-\tfrac{9}{4} r + 2) x_{8}^{2} x_{20}^{2} \\
  & + \tfrac{1}{2} x_{8}^{2} x_{21}^{2} + (\tfrac{9}{2} r - 4) x_{8} x_{9} x_{19} x_{20} + x_{8} x_{10} x_{14} x_{16} + x_{8} x_{10} x_{15} x_{17} \\
  & + (-\tfrac{9}{2} r + 3) x_{8} x_{10} x_{18} x_{20} - x_{8} x_{10} x_{19} x_{21} + \tfrac{1}{2} x_{9}^{2} x_{11}^{2} + \tfrac{1}{2} x_{9}^{2} x_{12}^{2} + \tfrac{3}{4} r x_{9}^{2} x_{13}^{2} + \tfrac{1}{2} x_{9}^{2} x_{15}^{2} \\
  & + \tfrac{1}{2} x_{9}^{2} x_{16}^{2} + \tfrac{1}{2} x_{9}^{2} x_{18}^{2} + (-\tfrac{9}{4} r + 2) x_{9}^{2} x_{19}^{2} + \tfrac{1}{2} x_{9}^{2} x_{21}^{2} - x_{9} x_{10} x_{14} x_{15} + x_{9} x_{10} x_{16} x_{17} \\
  & + (\tfrac{9}{2} r - 3) x_{9} x_{10} x_{18} x_{19} - x_{9} x_{10} x_{20} x_{21} + \tfrac{1}{2} x_{10}^{2} x_{11}^{2} + \tfrac{1}{2} x_{10}^{2} x_{12}^{2} + \tfrac{3}{4} r x_{10}^{2} x_{13}^{2} + \tfrac{1}{2} x_{10}^{2} x_{14}^{2} \\
  & + \tfrac{1}{2} x_{10}^{2} x_{17}^{2} + (-\tfrac{9}{4} r + 2) x_{10}^{2} x_{18}^{2} + \tfrac{1}{2} x_{10}^{2} x_{19}^{2} + \tfrac{1}{2} x_{10}^{2} x_{20}^{2}.
\end{align*}
}

\bibliography{/home/silvio/Zotero/biblio.bib}
\bibliographystyle{amsalpha}

\end{document}